\documentclass[11pt,a4paper]{amsart}[2000]
\title[The Rokhlin property for endomorphisms]{The Rokhlin property for endomorphisms and strongly self-absorbing $C^*$-algebras}

\author{Jonathan Brown}
\address{Department of Mathematics, University of Dayton, 300 College Park, Dayton, OH 45469-2316, USA}
\email{jonathan.henry.brown@gmail.com}
\author{Ilan Hirshberg}
\address{Department of Mathematics, Ben Gurion University of the Negev, P.O.B. 653, Be'er Sheva 84105, Israel}
\email{ilan@math.bgu.ac.il}
\thanks{This research was supported by Israel Science Foundation grant 1471/07 and by the Center for Advanced Studies in Mathematics 
at Ben-Gurion University of the Negev.}


\usepackage{amssymb}
\usepackage{amsthm}
\usepackage[all]{xypic}

\theoremstyle{plain}

\newtheorem{Thm}{Theorem}[section]

\newtheorem{Lemma}[Thm]{Lemma}

\newtheorem{Prop}[Thm]{Proposition}

\theoremstyle{definition}
\newtheorem{Def}[Thm]{Definition}
\newtheorem{Notation}[Thm]{Notation}

\newcommand{\B}{B}
\newcommand{\A}{A}

\newcommand{\D}{D}

\newcommand{\Zh}{\mathcal{Z}}
\newcommand{\E}{E}
\newcommand{\Oh}{\mathcal{O}}

\newcommand{\N}{{\mathbb N}}
\newcommand{\Z}{{\mathbb Z}}

\newcommand{\aut}{\mathrm{Aut}}

\newcommand{\eps}{\varepsilon}
\numberwithin{equation}{section}

\newcommand{\Alim}{\underrightarrow{\A}}
\newcommand{\alphalim}{\underrightarrow{\alpha}}

\begin{document}
\begin{abstract}
In this paper we define a Rokhlin property for automorphisms of non-unital $C^*$-algebras and for endomorphisms. We show that the crossed product of a $C^*$-algebra by a Rokhlin automorphism preserves  absorption of a strongly self-absorbing $C^*$-algebra, and use this result to deduce that the same result holds for crossed products by endomorphisms in the sense of Stacey. This generalizes earlier results of the second named author and W.~ Winter.
\end{abstract}
\maketitle

The goal of this paper is to generalize results from \cite{HW} to the cases of automorphisms of non-unital $C^*$-algebras and of crossed products by endomorphisms in the sense of Stacey. It was shown in \cite{HW} that given a Rokhlin automorphism $\alpha$ of a unital $C^*$-algebra $A$, if $A$ absorbs a
strongly self-absorbing $C^*$-algebra $D$ then the crossed product $A\rtimes_\alpha \mathbb{Z}$ also
absorbs $D$.  Strongly self-absorbing $C^*$-algebras are those infinite dimensional separable unital
$C^*$-algebras $D$  such that the map $d\mapsto d\otimes 1$ from $D$ to $D\otimes D$ is approximately
unitarily equivalent to an isomorphism. We refer the reader to \cite{TW} for the basic theory of
strongly self-absorbing $C^*$-algebras. The known examples of strongly self-absorbing
$C^*$-algebras are the Jiang-Su algebra $\mathcal{Z}$, the Cuntz algebras $\mathcal{O}_2$ and
$\mathcal{O}_\infty$, UHF algebras of  infinite type, and tensor products of
$\mathcal{O}_\infty$ by such UHF-algebras. We say that $A$ is $D$-absorbing if $A \cong A \otimes D$. The property of $D$-absorption for various strongly self-absorbing $C^*$-algebras
$D$ plays an important role in structure and classification theory of $C^*$-algebras. For
Kirchberg algebras $A$, the fundamental theorems in \cite{kirchberg-phillips} state that $A \otimes \Oh_{\infty} \cong A$ and $A \otimes \Oh_2 \cong \Oh_2$, and those are used in the Kirchberg-Phillips classification (\cite{phillips-classification}). 
In the stably finite case, $D$-absorption is central to Winter's classification theorem (\cite{winter-localizing}). We refer the reader to \cite{Ror91, Ror92} and \cite{GJS00, Ror04} for regularity properties of UHF-absorbing and $\Zh$-absorbing $C^*$-algebras, respectively, and to \cite{elliott-toms} for a survey of regularity properties in the context of classification of $C^*$-algebras.

In \cite{Cuntz}, Cuntz  introduced a notion of crossed product by an endomorphism to aid in the study of the Cuntz algebras. In particular he shows that the Cuntz algebras are crossed products of UHF algebras by an endomorphism.  This crossed product was later developed by Paschke \cite{Pas80} and Stacey \cite{Sta93}. 
Stacey showed that crossed products by endomorphisms are universal for certain representations and (following Cuntz) that they can be realized as a corner of a crossed product by an automorphism.  

In this note we shall first define a suitable version of the Rokhlin property for automorphisms of non-unital $C^*$-algebras (Definition~\ref{nonunital Rokhlin}) and show that if $\alpha$ is a Rokhlin automorphism of a $C^*$-algebra $A$ and $A$ absorbs a strongly self-absorbing $C^*$-algebra  $D$ then so does $A\rtimes_\alpha \Z$ (Theorem~\ref{Thm:permanence-auto}). We then define a Rokhlin property for endomorphisms of $C^*$-algebras (Definition~\ref{Def:Rokhlin-endo}) and use Stacey's construction and Theorem~\ref{Thm:permanence-auto} to prove that crossed products by Rokhlin endomorphisms also preserve $D$-absorption (Proposition~\ref{Prop:endo}).

\section{Automorphisms - the non-unital case}

\begin{Notation}
Let $\A$ be a $C^*$-algebra. As usual, we denote $\A_{\infty} = \ell^{\infty}(\N,\A)/c_0(\N,\A)$, and we identify $\A$ with the image of the subalgebra of $\ell^{\infty}(\N,\A)$ consisting of the constant sequences in the quotient. If $G$ is a discrete group and $\alpha:G \to \aut(\A)$ is an action, then we have a naturally induced action of $G$ on $\ell^{\infty}(\N,\A)$, which descends to an action on $\A_{\infty}$ leaving the central sequence algebra $\A_{\infty} \cap \A'$ invariant. We denote those induced actions by $\bar{\alpha}$. The case of interest in this paper will be of $G = \Z$, in which case we denote throughout $\alpha = \alpha_1$. 
\end{Notation}

\begin{Def} \label{nonunital Rokhlin}
Let $\A$ be a $C^*$-algebra, and $\alpha \in \aut(\A)$. We say that $\alpha$ has the \emph{Rokhlin property} if for any positive integer $p$, any finite set $F \subset \A$,  and any $\eps>0$, there are mutually orthogonal positive contractions  
$e_{0,0},\ldots,e_{0,p-1}, e_{1,0},\ldots,e_{1,p} \in \A$ such that
\begin{enumerate}
\item $\left\|\left ( \displaystyle \sum_{r=0}^1\sum_{j=0}^{p-1+r} e_{r,j} \right ) a - a \right\|<\eps$ for all $a \in F$,
\item $\left\|[e_{r,j},a]\right\| < \eps$ for all $r,j$ and $a \in F$,
\item $\left\|\alpha(e_{r,j})a - e_{r,j+1}a\right\|<\eps$ for all $a\in F$, $r=0,1$ and  $j=0,1,\ldots,p-2+r$, and
\item  $\left\|\alpha(e_{0,p-1} + e_{1,p})a - (e_{0,0} + e_{1,0})a\right\|<\eps$ for all $a \in F$.
\end{enumerate} 
\end{Def}
This will be referred to as \emph{Rokhlin dimension zero} in \cite{HP}, generalizing the definitions from \cite{HWZ}.

We note that although the elements $e_{r,j}$ are not required to be projections, they behave almost like projections when multiplied by elements from $a$: 
$$
\|(e_{r,j}^2 - e_{r,j})a\|<\eps .
$$
This follows immediately from the fact that $\displaystyle e_{r,j}^2  = e_{r,j} \sum_{k=0}^1\sum_{j=0}^{p-1+k} e_{k,j}$.

We can reformulate this property in terms of the central sequence algebra as well. The proof is straightforward and will be omitted.

\begin{Lemma}
\label{Lemma:central-sequence-reformulation}
Let $\A$ a separable $C^*$-algebra and $\alpha \in \aut (\A)$.
Then $\alpha$ has the Rokhlin property if and only if for any positive integer $p$ and any separable self-adjoint subspace $E \subseteq \A_{\infty}$ there are mutually orthogonal positive contractions 
$e_{0,0},\ldots,e_{0,p-1}, e_{1,0},\ldots,e_{1,p}$
in $ \A_{\infty} \cap (\A + E)'$
such that
\begin{enumerate}
\item $\left ( \displaystyle \sum_{r=0}^1\sum_{j=0}^{p-1+r} e_{r,j} \right ) a = a$ for all $a \in \A + E$,
\item $\bar{\alpha}(e_{r,j})a = e_{r,j+1}a$, for all $r,j$ for all $j=0,1,\ldots,p-2+r$, and all  $a \in \A + E$,
\item  $\bar{\alpha}(e_{0,p-1} + e_{1,p})a = (e_{0,0} + e_{1,0})a$ for all $a \in \A + E$.
\end{enumerate} 
\end{Lemma}

We recall that it was shown in \cite{winter-ssa-Z-stable} that any strongly self-absorbing $C^*$-algebra is $K_1$-injective. This was previously required as an extra assumption in some of the theorems we refer to. In particular, it follows from \cite[Proposition 1.13]{TW} that one can choose the unitaries implementing the approximately inner half-flip of $\D \otimes D$ to be in $U_0(\D \otimes \D)$, and we shall use this fact below in the proof of Theorem \ref{Thm:permanence-auto}.

\begin{Def}
\label{Def-morphism}
Let $\A$ be a separable $C^*$-algebra, let $\B$ be a unital $C^*$-algebra, and let $E \subset \A_{\infty}$ be a self-adjoint separable subspace. A c.p.c.~ map $\varphi: \B \to \A_{\infty}$ will be called a \emph{morphism relative to $E$} if $\varphi(\B) \subseteq \A_{\infty} \cap E'$ and 
\begin{enumerate}
\item $(\varphi(xy) -\varphi(x)\varphi(y))e = 0$ for all $e \in \E$ and all $x,y \in \B$, and
\item $\varphi(1)e = e$ for all $e \in E$.
\end{enumerate}
\end{Def}
The two numbered conditions in this definition can be reformulated as saying that the composition of $\varphi$ with the quotient map $\A_{\infty} \to \A_{\infty}/\mathrm{Ann}(E)$ is a unital homomorphism. The algebra $(\A_{\infty} \cap \A')/\mathrm{Ann}(\A)$ was considered in \cite{kirchberg} as a substitute for $\A_{\infty} \cap \A'$ when $\A$ is non-unital, but we shall not use this formalism here. Using Definition \ref{Def-morphism}, we now rephrase the part of \cite[Proposition 4.1]{HRW} which is relevant to us as follows.
\begin{Prop}
\label{Prop:HRW-4.1-variant}
Let $\A$ be a separable $C^*$-algebra and let $\D$ be a strongly self-absorbing $C^*$-algebra. $\A$ is $\D$-absorbing if and only if there is a morphism from $\D$ to $\A_{\infty}$ relative to $\A$.
\end{Prop}

We record the following simple lemma. The proof is a straight-forward diagonalization argument, which we omit (see \cite[Lemma 4.5]{HW} for a very similar statement and proof).

\begin{Lemma}
Let $\A,\B$ be separable $C^*$-algebras, with $\B$ unital. Suppose that there exists a morphism from $\B$ to $\A_{\infty}$ relative to $\A$, then for any separable self-adjoint subspace $E \subseteq \A_{\infty}$ there is a  morphism from $\B$ to $\A_{\infty}$ relative to $\A + E$.
\end{Lemma}

We shall require the following lemma, which is a variant of \cite[Lemma 2.3]{HW} better suited for our needs.
\begin{Lemma}
\label{Lemma:before-HW-2.4-variant}
Let $\A$ be a separable $C^*$-algebra. Let $G$ be a discrete countable group, and let $\alpha:G \to \aut(\A)$ be an action. Let $\D$ be a strongly self-absorbing $C^*$-algebra. Suppose that there is a morphism $\varphi$ from $\D$ to $\A_{\infty}$ relative to $\A$ that furthermore satisfies $\bar{\alpha}_g(\varphi(x))a = \varphi(x)a $ for all $a \in \A$ and all $g \in G$, then the full crossed product $A \rtimes_{\alpha}G$ is $\D$-absorbing (and hence all other crossed products).
\end{Lemma}
\begin{proof}
The canonical inclusion $\A \hookrightarrow \A\rtimes_{\alpha}G$ induces an inclusion $\A_{\infty} \hookrightarrow (\A\rtimes_{\alpha}G)_{\infty}$. Pick a map $\varphi$ as in the statement. Composing with this canonical inclusion (and retaining the same notation), we can view $\varphi$ as a morphism from $\D$ to $(\A\rtimes_{\alpha}G)_{\infty}$ relative to $\A$. However, if $u_g$ is one of the canonical unitaries in $M(\A\rtimes_{\alpha}G)$, $a \in \A$ and $x \in \D$ then we have that 
$$
a u_g \varphi(x) = a \bar{\alpha}_g(\varphi(x)) u_g = a \varphi(x) u_g = \varphi(x) a u_g
$$
and since the elements of the form $a u_g$ span $\A \rtimes_{\alpha}G$, we find that $\varphi$ is a morphism relative to $\A \rtimes_{\alpha}G$. Thus, by Proposition \ref{Prop:HRW-4.1-variant}, $\A \rtimes_{\alpha}G$ is $\D$-absorbing.
\end{proof}

We shall also require the following variant of \cite[Lemma 2.4]{HW}. The idea of the proof is to show that the conditions in this lemma imply the conditions of Lemma \ref{Lemma:before-HW-2.4-variant}. The proof is a straightforward modification of the argument in \cite{HW} and will be omitted.
\begin{Lemma}
\label{Lemma:HW-2.4-variant}
Let $\A$ be a separable $C^*$-algebra. Let $G$ be a discrete countable group. Let $S \subseteq G$ be a generating subset. Let $\alpha:G \to \aut(\A)$ be an action. Let $\D$ be a strongly self-absorbing $C^*$-algebra. Suppose that for any $\eps>0$ and any finite subset $G_0 \subseteq S$, any finite subset $A_0 \in \A$ and any finite subset $D_0 \in \D$ of elements in the unit ball there is a c.p.c.~ map $\varphi$ from $\D$ to $\A_{\infty}$ which satisfies 
\begin{enumerate}
\item $\|(\varphi(xy) - \varphi(x)\varphi(y))a\| < \eps$ for all $x,y \in D_0$, $a \in A_0$,
\item $\|\varphi(1)a - a\|<\eps$ for all $a \in A_0$,
\item $\|[\varphi(x),a]\|<\eps$ for all $x \in D_0$, $a \in A_0$, and
\item $\|(\bar{\alpha}_g(\varphi(x)) - \varphi(x))a\|<\eps$ for all $x \in D_0$, $a \in A_0$, and all $g \in G_0$,
\end{enumerate}
then the full crossed product $A \rtimes_{\alpha}G$ is $\D$-absorbing (and hence all other crossed products).
\end{Lemma}

\begin{Thm}
\label{Thm:permanence-auto}
Let $\D$ be a strongly self-absorbing $C^*$-algebra. Let $\A$ be a separable $\D$-absorbing $C^*$-algebra and $\alpha \in \aut(\A)$
be an automorphism with the Rokhlin property, then  $\A \rtimes_{\alpha} \Z$ is $\D$-absorbing as well.
\end{Thm}
\begin{proof}

Fix a finite set $K$ in the unit ball of $\D$ and fix $\eps>0$.

We define c.p.c. maps $\iota,\mu:\D \to \A_{\infty} \cap \A '$ such that $\iota$ is a morphism relative to $\A$, and $\mu$ is a morphism relative to $\A + \mathrm{span}(\bigcup_{n \in \Z} \bar{\alpha}^n(\iota(\D)))$.  

Choose a unitary $w \in U_0(\D \otimes \D)$ such that
$$\|w(x\otimes 1)w^*-1 \otimes x\| < \frac{\eps}{4}$$ for all $x \in K$. 
Thus, $w$ can be connected to $1$ via a rectifiable path.
Let $L$ be the length of such a path. Choose $n$ such that
$L\|x\|/n <\eps/8$ for all $x \in F$.

Define c.p. maps 
$$
\rho,\rho':\D \otimes \D \to
\A_{\infty} \cap \A' 
$$ 
by
$$
\rho(x \otimes y) = \iota (x) \mu(y), \;\; \rho'(x\otimes y) =
\bar{\alpha}^n(\iota(x))\mu(y) \, .
$$ 
Since the image of $\mu$ commutes with the image of $\iota$ and its images under iterates of $\bar{\alpha}$, those maps are well defined. We note that $\rho(1)a = a$ and $(\rho(xy) - \rho(x)\rho(y))a =0$ for all $a \in \A$ and $x,y \in \D \otimes \D$, and the same holds for $\rho'$ as well. This evidently holds for elementary tensors in $\D \otimes \D$, and by linearity and continuity holds for all $x,y \in \D \otimes \D$.

Pick unitaries $1=w_0,w_1,\ldots,w_n = w \in
U_0(\D \otimes \D)$ such that $\|w_j-w_{j+1}\| \leq L/n$ for
$j=0,\ldots,n-1$. Now, for $k=0,\ldots,d$, let 
$$
x_j = \rho(w_j)^*\rho'(w_j);
$$ 
the $x_j$'s behave like unitaries when multiplied by elements from $\A$, i.e. we have
$$
x_jx_j^*a = x_j^*x_ja = a
$$
for all $a \in \A$. Furthermore, $x_0a = a$ for all $a \in \A$. 
Note also that
$$\|x_j - x_{j+1}\| \leq \frac{2L}{n}$$ for $j=0,\ldots,n-1$, and  
that $$\|\left ( x_n \bar{\alpha}^n(\iota(x)) x_n^{*} - \iota(x) \right ) a\| <
\frac{\eps}{2}$$ for all $x\in K$ and all $a \in \A$ with norm at most 1.

Likewise, pick unitaries $1=w'_0,w'_1,\ldots,w'_{n+1} = w' \in
U_0(\D \otimes \D)$ such that $\|w_j-w_{j+1}\| \leq L/{n+1}$ for
$j=0,\ldots,n$, and let 
$$
y_j = \rho_k(w_j')^*\rho'(w_j').
$$
The $y_j$'s satisfy the analogous properties to those of $x_j$'s, where $n$ is replaced by $n+1$.

Let $E$ be the (separable) $C^*$-subalgebra of $\A_\infty \cap \A'$ generated by $\bar{\alpha}^j(\iota(\D))$ and $\bar{\alpha}^j(\mu(\D))$ for all $j$.
Let $\{e_{0,0},\ldots,e_{0,n-1}, e_{1,0},\ldots,e_{1,n} \} \in \A_{\infty} \cap (\A+E) '$ 
be Rokhlin elements in $\A_{\infty} \cap (\A+E) '$  as  in Lemma \ref{Lemma:central-sequence-reformulation}.

Now, define $\theta:\D \to \A_{\infty} \cap \A'$ by 
$$
\theta (x)= \sum_{j=0}^{n-1}e_{0,j}\bar{\alpha}^{j-n}(x_j)\bar{\alpha}^j
(\iota(x))\bar{\alpha}^{j-n}(x_j^{*}) + 
\sum_{j=0}^{n}e_{1,j}\bar{\alpha}^{j-n}(y_j)\bar{\alpha}^j
(\iota(x))\bar{\alpha}^{j-n}(y_j^{*}).
$$

We check that
\begin{enumerate}
\item $(\theta(xy)-\theta(x)\theta(y))a = 0$ for all $x,y \in \D$, $a \in \A$,
\item $\theta(1)a = a$ for all $a \in \A$, and
\item $\|(\bar{\alpha}(\theta(x)) - \theta(x)) a\|<\eps$ for all $x \in K$. 
\end{enumerate}
Thus, $\theta$ satisfies the conditions of Lemma \ref{Lemma:HW-2.4-variant}, so $\A \rtimes_{\alpha}\Z$ is $\D$-absorbing.
\end{proof}

\section{Endomorphisms}

We now consider the case of endomorphisms. 

\begin{Def} \label{Def:Rokhlin-endo}
Let $\A$ be a $C^*$-algebra, and let $\alpha: \A \to \A$ be an endomorphism. We say that $\alpha$ has the \emph{Rokhlin property} if for any positive integer $p$, any finite set $F \subset \alpha^p(\A)$, 
and any $\eps>0$, there are mutually orthogonal positive contractions  
$e_{0,0},\ldots,e_{0,p-1}, e_{1,0},\ldots,e_{1,p}$ such that
\begin{enumerate}
\item $\left\|\left ( \displaystyle \sum_{r=0}^1\sum_{j=0}^{p-1+r} e_{r,j} \right ) a - a \right\|<\eps$ for all $a \in F$,
\item $\left\|[e_{r,j},a]\right\| < \eps$ for all $r,j$ and $a \in F$,
\item $\left\|\alpha(e_{r,j})a - e_{r,j+1}a\right\|<\eps$ for all $a\in F$, $r=0,1$, and $j=0,1,\ldots,p-2+r$,
\item  $\left\|\alpha(e_{0,p-1} + e_{1,p})a - (e_{0,0} + e_{1,0})a\right\|<\eps$ for all $a \in F$.
\end{enumerate} 
\end{Def}
This definition is (formally) weaker than \cite[Definition 4.1]{rordam-95}. We do not know whether they coincide when $A$ is unital (in \cite[Definition 4.1]{rordam-95} $A$ is assumed to be unital).
One can define higher Rokhlin dimension in a similar way, however we shall not pursue it further here.

Examples of endomorphisms satisfying the Rokhlin property (already in R{\o}rdam's sense) play an important role in the theory. It was shown in \cite[Corollary 4.6]{rordam-95} that if $G_0,G_1$ are countable abelian groups with $G_1$ torsion free, and $g_0 \in G_0$ then there is a simple unital AF algebra $A$ and an endomorphism $\alpha$ of $A$ satisfying the Rokhlin property such that $A \rtimes_{\alpha} \N$ is a unital Kirchberg algebra for which $(K_0(A \rtimes_{\alpha} \N),[1],(K_1(A \rtimes_{\alpha} \N)) \cong (G_0,g_0,G_1)$. For a more specific example, if we let $A \cong M_n \otimes M_n \otimes \cdots$, and let $e \in M_n$ be a minimal projection, then the endomorphism $\alpha : A \to A$ given by $\alpha(a_1 \otimes a_2 \otimes a_3 \otimes \cdots) = e \otimes a_1 \otimes a_2 \otimes a_3 \otimes \cdots$ is the original endomorphism considered in \cite{Cuntz}, which satisfies $A \rtimes_{\alpha} \N \cong \Oh_n$. The bilateral tensor shift on the infinite tensor power of $M_n$ satisfies the Rokhlin property (see \cite{BSKR} for the case of $n$ even, and \cite{kishimoto-odd} for the odd case). From this it follows that the endomorphism $\alpha$ described above has the Rokhlin property in our sense as well (see \cite[Example 4.3]{rordam-95}). To see that, denote by $\sigma$ the bilateral shift automorphism on $\bigotimes_{-\infty}^{\infty}M_n$. We think of $\A$ as the subalgebra $1^{\otimes \infty} \otimes \bigotimes_{0}^{\infty}M_n$ of the two-sided tensor power, and observe that $\alpha (a) = (e \otimes 1^{\otimes \infty})\cdot \sigma(a)$ for all $a \in \A$. Fix $\eps>0$, $p\in \N$ and $F \in \alpha^p(\A)$. Let $(p_{r,j})_{r=0,1 \, ; \, j=0,1,...,p-1+r}$ be Rokhlin projections for the automorphisms $\sigma$ with respect to the set $F$ and the given $\eps$. By perturbing the projections $p_{r,j}$ if need be, we may assume without loss of generality that they are all contained in a subalgebra of the form $1^{\otimes \infty} \otimes \bigotimes_{-m}^m M_n  \otimes 1^{\otimes \infty}$.  We may also replace the elements $p_{r,j}$ by $\sigma^{N}(p_{r,j})$ for all sufficiently large $N$ and they will still satisfy the same Rokhin conditions (the projections will still almost commute with elements of $F$, since  for any $x,y$ we have $\lim_{k \to \infty} \|[x,\sigma^k(y)]\| = 0$). In particular, we may assume without loss of generality that the projections $p_{r,j}$ lie in $1^{\otimes (p+1)} \otimes \bigotimes_{p+1}^{\infty}M_n$. Those projections now satisfy the conditions of Definition~\ref{Def:Rokhlin-endo} for the endomorphism $\alpha$.

If $\alpha:\A \to \A$ is an endomorphism, we define $\Alim$ to be the inductive limit of the stationary inductive system $\A \overset{\alpha}{\to} \A \overset{\alpha}{\to} \cdots$ (this is sometimes denoted in the literature by $\A_{\infty}$, however we already used this notation for the sequence algebra). The endomorphism $\alpha$ induces an automorphism of $\Alim$, given by the downwards diagonal arrows in the following commuting diagram.
$$
\xymatrix{
\A \ar[r]^{\alpha}  \ar[dr]^{\alpha} & \A \ar[r]^{\alpha}\ar@{=}[d] \ar[dr]^{\alpha}  & \A \ar[r]^{\alpha}\ar@{=}[d] \ar[dr]^{\alpha} & \cdots \\
 & \A \ar[r]^{\alpha} & \A \ar[r]^{\alpha} & \cdots
}
$$
We denote the resulting automorphism by $\alphalim$. We use this setup to pass from endomorphisms to the case of automorphisms in the non-unital setting, of Theorem~\ref{Thm:permanence-auto}. In the examples we describe above, permanence of $D$-absorption can be read off $K$-theory data, and in that sense our results give nothing new; however this method provides a direct and different way of obtaining it.

\begin{Prop}\label{Prop:endo} Let $\alpha: \A \to \A$ be an endomorphism satisfying the Rokhlin property.
\begin{enumerate}
\item \label{1} The automorphism $\alphalim$ has the Rokhlin property as an automorphism of $\Alim$.
\item \label{2} If $D$ is a strongly self-absorbing $C^*$-algebra and $A$ is a separable unital $D$-absorbing $C^*$-algebra  then $A\rtimes_{\alpha}\N$ is $D$-absorbing as well.
\end{enumerate}
\end{Prop}
\begin{proof}
For \eqref{1}, let $\gamma_n : \A \to \Alim$ be the canonical homomorphisms as in the definition of an inductive limit, i.e. $\gamma_{n+1} \circ \alpha = \gamma_n$ and $\Alim = \overline{\bigcup_n \gamma_n(\A)}$, where the union is increasing. Let $F \subseteq \Alim$ be a given finite set, let $\eps>0$ and let $p>0$ be a given integer. We may assume without loss of generality that  $F \subseteq \bigcup_n \gamma_n(\A)$. Fix $N>p$, and choose a finite set $F_N \subseteq \A$ such that $\gamma_{N}(F_N) = F$. Find elements $e_{r,i}$ as in Definition~\ref{Def:Rokhlin-endo}, and one now checks that the elements $\gamma_N(e_{r,i})$ satisfy the conditions of Definition \ref{nonunital Rokhlin} with respect to the given $(F,\eps)$. 

For \eqref{2}, by \cite[Proposition 3.3]{Sta93} there exists a projection $p\in M(\Alim\rtimes_{\alphalim}\Z)$ such that $A\rtimes_{\alpha}\N=p(\Alim\rtimes_{\alphalim}\Z) p$.
Thus by \cite[Corollary 3.1]{TW}, $A\rtimes_{\alpha}\N$ is $D$-absorbing if $\Alim\rtimes_{\alphalim}\Z$ is.  Theorem~\ref{Thm:permanence-auto} and \eqref{1} now give the result. \end{proof}

\end{document}